\documentclass[11pt,reqno]{amsart}

\newtheorem{proposition}{Proposition}[section]
\newtheorem{lemma}[proposition]{Lemma}

\newtheorem{theorem}[proposition]{Theorem}

\theoremstyle{definition}
\newtheorem{definition}[proposition]{Definition}
\newtheorem{example}[proposition]{Example}
\newtheorem{examples}[proposition]{Examples}

\newtheorem{remarks}[proposition]{Remarks}

\newcommand{\thlabel}[1]{\label{th:#1}}
\newcommand{\thref}[1]{Theorem~\ref{th:#1}}
\newcommand{\selabel}[1]{\label{se:#1}}
\newcommand{\seref}[1]{Section~\ref{se:#1}}
\newcommand{\lelabel}[1]{\label{le:#1}}
\newcommand{\leref}[1]{Lemma~\ref{le:#1}}

\newcommand{\relabel}[1]{\label{re:#1}}

\newcommand{\exlabel}[1]{\label{ex:#1}}
\newcommand{\exref}[1]{Example~\ref{ex:#1}}
\newcommand{\delabel}[1]{\label{de:#1}}
\newcommand{\deref}[1]{Definition~\ref{de:#1}}
\newcommand{\eqlabel}[1]{\label{eq:#1}}
\newcommand{\equref}[1]{(\ref{eq:#1})}

\def\lan{\langle}
\def\ran{\rangle}
\def\ot{\otimes}

\newcommand{\Cc}{\mathcal{C}}

\def\*C{{}^*\hspace*{-1pt}{\Cc}}
\def\text#1{{\rm {\rm #1}}}

\input xy
\xyoption {all} \CompileMatrices

\usepackage{amssymb}
\usepackage{color,amssymb,graphicx,amscd,amsmath}
\usepackage[colorlinks,urlcolor=blue,linkcolor=blue,citecolor=blue]{hyperref}

\begin{document}

\title[Classifying complements]
{Classifying complements for Hopf algebras and Lie algebras}

\author{A. L. Agore}
\address{Faculty of Engineering, Vrije Universiteit Brussel, Pleinlaan 2, B-1050 Brussels, Belgium}
\email{ana.agore@vub.ac.be and ana.agore@gmail.com}

\author{G. Militaru}
\address{Faculty of Mathematics and Computer Science, University of Bucharest, Str.
Academiei 14, RO-010014 Bucharest 1, Romania}
\email{gigel.militaru@fmi.unibuc.ro and gigel.militaru@gmail.com}

\thanks{A.L. Agore is research fellow ''Aspirant'' of FWO-Vlaanderen.
This work was supported by a grant of the Romanian National
Authority for Scientific Research, CNCS-UEFISCDI, grant no.
88/05.10.2011.}



\begin{abstract}
Let $A \subseteq E$ be a given extension of Hopf (respectively
Lie) algebras. We answer the \emph{classifying complements
problem} (CCP) which consists of describing and classifying all
complements of $A$ in $E$. If $H$ is a given complement then all
the other complements are obtained from $H$ by a certain type of
deformation. We establish a bijective correspondence between the
isomorphism classes of all complements of $A$ in $E$ and a
cohomological type object ${\mathcal H}{\mathcal A}^{2} (H, A \, |
\, (\triangleright, \triangleleft) )$, where $(\triangleright,
\triangleleft)$ is the matched pair associated to $H$. The
factorization index $[E: A]^f$ is introduced as a numerical
measure of the (CCP). For two $n$-th roots of unity we construct a
$4n^2$-dimensional Hopf algebra whose factorization index over the
group algebra is arbitrary large.
\end{abstract}

\maketitle


\section*{Introduction}
Let $\Cc$ be the category of groups, Lie algebras, Hopf algebras,
etc. and $A \subset E$ a given subobject of $E$ in $\Cc$. A
subobject $H$ of $E$ is called a \emph{complement} of $A$ in $E$
(or an \emph{$A$-complement} of $E$) if $E$ can be written as a
'product' of $A$ and $H$ such that $A$ and $H$ have 'minimal
intersection' in $E$; the meaning of 'product' and 'minimal
intersection' depends on the given category $\mathcal{C}$. We
denote by $[E: A]^f$ the cardinal of the (possibly empty)
isomorphism classes of all $A$-complements of $E$ and we call it
the \emph{factorization index} of $A$ in $E$. A natural question
arises:

\textbf{Classifying complements problem (CCP):} \textit{Let $A
\subset E$ be a given subobject of $E$ in $\Cc$. If an
$A$-complement of $E$ exists, describe explicitly, classify all
$A$-complements of $E$ and compute the factorization index $[E:
A]^f$.}

To start with, let $\Cc = {\mathcal Gr}$ be the category of
groups. An $A$-complement of a group $E$ is a subgroup $H \leq E$
such that $E = A H$ and $A \cap H = \{1\}$. An $A$-complement of
$E$, if exists, is not necessarily unique. The basic example is
the following: consider $S_3$ to be the symmetric group viewed as
a subgroup in $S_4$ by considering $4$ to be a fixed point. Then,
the factorization index $[S_4 : S_3]^f = 2$. For more details on
the group case we refer the reader to \cite{am-2012}. Hence, we
expect to obtain non-trivial results for the (CCP) whose
difficulty depends on the category $\Cc$ as well as on the
extension $A \subset E$ in $\Cc$.

The aim of this paper is to give the answer to the (CCP) if $\Cc$
is the category of Hopf (respectively Lie) algebras. There exists
a general principle: $H$ is an $A$-complement of $E$ in a given
category $\Cc$ if and only if $E \cong A \bowtie H$, where $A
\bowtie H$ is a '\emph{bicrossed product}' in the category $\Cc$
associated to a '\emph{matched pair}' between the objects $A$ and
$H$. This principle becomes a theorem when $\Cc$ is the category
of groups or groupoids \cite{AA}, algebras \cite{cap}, Hopf
algebras \cite{majid2}, Lie groups or Lie algebras \cite{majid3},
locally compact quantum groups \cite{VV}, multiplier Hopf algebras
\cite{DeDW}. Let $H$ be a given $A$-complement of $E$. Hence,
there exists a canonical isomorphism $A\bowtie H \cong E$ in
$\Cc$. Now, the description and the classification part of the
(CCP) is obtained from the following subsequent question: describe
and classify all objects ${\mathbb H}$ in $\Cc$ such that there
exists an isomorphism $A \bowtie H \cong A \bowtie {\mathbb H}$ in
$\Cc$. This can also be viewed as a descent type question: the
classification of all $A$-complements of $E$ needs a parallel
theory similar to what is called the classification of forms in
the classical descent theory \cite{KO}, \cite{parker}.

As we will see, the theoretical answers to the (CCP) for Hopf
algebras and Lie algebras follow a common path which might
generate a method of tackling this problem in different other
categories. In order to describe and then classify all the
complements of a given extension $A \subset E$ it is enough to
know only one complement $H$. All the other $A$-complements of $E$
are obtained from $H$ by a certain type of deformation which
involves some special maps $r: H \to A$ associated to the
canonical matched pair $(A, H, \triangleleft, \triangleright)$,
called deformation maps. To each deformation map $r: H \to A$ we
associate a new $A$-complement denoted by $H_{r}$ and, conversely,
for any $A$-complement ${\mathbb H}$ there exists a deformation
map $r: H \to A$ of the canonical matched pair $(A, H,
\triangleleft, \triangleright)$ such that ${\mathbb H} \cong H_r$.

The paper is organized as follows. In \seref{prel} we shall review
the Majid's bicrossed product \cite{majid} associated to a matched
pair of Hopf algebras $(A, H, \triangleleft, \triangleright)$.
\seref{3aplicatii} offers the answer to the (CCP) problem for Hopf
algebras. The answer will be given in three steps. In
\thref{deformatANM} a new type of deformation of a given Hopf
algebra $H$ is introduced. This deformation $H_r$ is a new Hopf
algebra called the \emph{$r$-deformation} of $H$ and is associated
to an arbitrary matched pair of Hopf algebras $(A, H,
\triangleright, \triangleleft)$ and to a deformation map $r: H \to
A$ in the sense of \deref{amisur}. Furthermore, $H_r$ is an
$A$-complement of the bicrossed product $A \bowtie H$. Now, let $A
\subseteq E$ be an extension of Hopf algebras and $H$ a given
$A$-complement of $E$. The description of all $A$-complements of
$E$ is given in \thref{descformelor}: any other $A$-complement
${\mathbb H}$ of $E$ is isomorphic as a Hopf algebra with an
$H_{r}$, for some deformation map $r : H \to A$ of the canonical
matched pair  $(A, H, \triangleleft, \triangleright)$ associated
to $H$. Finally, \thref{clasformelor} provides the classification
of $A$-complements of $E$: there exists a bijection between the
isomorphism classes of all $A$-complements of $E$ and a
cohomological type object ${\mathcal H}{\mathcal A}^{2} (H, A \, |
\, (\triangleright, \triangleleft) )$ and this bijection is
explicitly described. In particular, we obtain the formula for
computing the factorization index of a given extension $A
\subseteq E$ of Hopf algebras: $[E : A]^f = | {\mathcal
H}{\mathcal A}^{2} (H, A \, | \, (\triangleright, \triangleleft)
)|$. In \seref{secexemple} we shall construct an example of a Hopf
algebra extension of a given factorization index in
\thref{findmare}. This is the extension $k[C_n] \subseteq H_{4n^2,
\, \omega, \, \omega'}$, where $H_{4n^2, \, \omega, \, \omega'}$
is a $4n^2$-dimensional quantum group associated to two distinct
$n$-th roots of unity $\omega$ and $\omega'$. In \seref{deflie} we
give the answer to the (CCP) in the case when $\Cc$ is the
category of Lie algebras. The main result of the section is
\thref{clasformelorLie}: if $\mathfrak{g}$ is a Lie subalgebra of
$\Xi$ and $\mathfrak{h}$ is a fixed $\mathfrak{g}$-complement of
$\Xi$, then the isomorphism classes of all
$\mathfrak{g}$-complements of $\Xi$ are parameterized by a certain
cohomological object denoted by ${\mathcal H}{\mathcal A}^{2}
(\mathfrak{h}, \mathfrak{g} \, | \, (\triangleright,
\triangleleft) )$ which is explicitly constructed.

\section{Preliminaries}\selabel{prel}
Unless explicitly specified otherwise, all coalgebras, Hopf
algebras, Lie algebras, tensor products, homomorphisms, and so on
are over a commutative ring $k$. For a coalgebra $C$, we use
Sweedler's $\Sigma$-notation: $\Delta(c) = c_{(1)}\ot c_{(2)}$,
$(I\ot\Delta)\Delta(c) = c_{(1)}\ot c_{(2)}\ot c_{(3)}$ (summation
understood). A Hopf subalgebra of a Hopf algebra $E$ is a Hopf
algebra inclusion $A \subseteq E$ that splits $k$-linearly. Let
$A$ and $H$ be two Hopf algebras. $H$ is called a right $A$-module
coalgebra if $H$ is a coalgebra in the monoidal category
${\mathcal M}_A $ of right $A$-modules, i.e. there exists
$\triangleleft : H \otimes A \rightarrow H$ a morphism of
coalgebras such that $(H, \triangleleft) $ is a right $A$-module.
Similarly, $A$ is a left $H$-module coalgebra if $A$ is a
coalgebra in the monoidal category of left $H$-modules.

Let $A \subseteq E$ be a Hopf subalgebra of $E$. A Hopf subalgebra
$H\subseteq E$ is called a \emph{right complement} of $A$ in $E$
(or a \emph{right $A$-complement} of $E$) if the multiplication
map $A \ot H \to E$, $a \ot h \mapsto a \, h $ is bijective.
Similarly, $H$ is a left complement of $A$ in $E$ if the
multiplication map $H \ot A \to E$ is bijective. In the case that
$E$ has a bijective antipode, then the concepts of right/left
$A$-complement coincide \cite[Proposition 3.1]{abm1}. Throughout
this paper by an $A$-complement we mean a right $A$-complement.

Let $H$ be a (right) $A$-complement of $E$: in this case we say
that the Hopf algebra $E$ factorizes through $A$ and $H$. The
bicrossed product of two Hopf algebras was introduced by Majid in
\cite[Proposition 3.12]{majid} under the name of double cross
product. Let $A$ and $H$ be two Hopf algebras and $\triangleleft :
H \otimes A \rightarrow H$, $\triangleright: H \otimes A
\rightarrow A$ two morphisms of coalgebras such that the following
normalizing conditions hold for any $h \in H$, $a \in A$:
\begin{equation}\eqlabel{mp1}
h \triangleright1_{A} = \varepsilon_{H}(h)1_{A}, \quad 1_{H} \triangleright a = a,
\quad 1_{H} \triangleleft a = \varepsilon_{A}(a)1_{H}, \quad h\triangleleft 1_{A} =
h
\end{equation}
We denote by $A \bowtie H$ the $k$-module $A \otimes H$ together
with the multiplication:
\begin{equation}\eqlabel{0010}
(a \bowtie h) \cdot (c \bowtie g):= a (h_{(1)}\triangleright c_{(1)})
\bowtie (h_{(2)} \triangleleft c_{(2)}) g
\end{equation}
for all $a$, $c\in A$, $h$, $g\in H$, where we denoted $a\ot h$ by
$a\bowtie h$. The object $A \bowtie H$ is called the \textit{
bicrossed product} of $A$ and $H$ if $A \bowtie H$ is a Hopf
algebra with the multiplication given by \equref{0010}, the unit
$1_{A} \bowtie 1_{H}$ and the coalgebra structure given by the
tensor product of coalgebras. The next theorem provides necessary
and sufficient conditions for $A \bowtie H$ to be a bicrossed
product.

\begin{theorem}\thlabel{bicrsreconst}
Let $A$, $H$ be two Hopf algebras and $\triangleleft : H \otimes A
\rightarrow H$, $\triangleright: H \otimes A \rightarrow A$ two
morphisms of coalgebras satisfying the normalizing conditions
\equref{mp1}. The following statements are equivalent:

$(1)$ $A \bowtie H$ is a bicrossed product;

$(2)$ $(H, \triangleleft)$ is a right $A$-module coalgebra, $(A, \triangleright)$ is
a left $H$-module coalgebra and the following compatibilities hold
for any $a$, $b\in A$, $g$, $h\in H$.
\begin{eqnarray}
g \triangleright(ab) &{=}& (g_{(1)} \triangleright a_{(1)}) \bigl ( (g_{(2)}\triangleleft
a_{(2)})\triangleright b \bigl)
\eqlabel{mp2} \\
(g h) \triangleleft a &{=}& \bigl( g \triangleleft (h_{(1)} \triangleright a_{(1)}) \bigl)
(h_{(2)} \triangleleft a_{(2)})
\eqlabel{mp3} \\
g_{(1)} \triangleleft a_{(1)} \otimes g_{(2)} \triangleright a_{(2)} &{=}& g_{(2)}
\triangleleft a_{(2)} \otimes g_{(1)} \triangleright a_{(1)} \eqlabel{mp4}
\end{eqnarray}
\end{theorem}

\begin{proof}
$(2) \Rightarrow (1)$ This is just \cite[Theorem 7.2.2]{majid2} or
\cite[Theorem IX 2.3]{Kassel}.

$(1) \Rightarrow (2)$ Follows as a special case of \cite[Theorem
2.4]{am-2011} if we consider $f : H \ot H \to A$ to be the trivial
cocycle, i.e. $f (g, h) = \varepsilon_H(g) \varepsilon_H(h)$. See
also \cite[Examples 2.5]{am-2011} for details.
\end{proof}

A \emph{matched pair} of Hopf algebras is a system $(A, H,
\triangleleft, \triangleright)$, where $(H, \triangleleft)$ is a
right $A$-module coalgebra, $(A, \triangleright)$ is a left
$H$-module coalgebra such that the compatibility conditions
\equref{mp1} and \equref{mp2}-\equref{mp4} hold.

\begin{examples} \exlabel{exempleban}
1. Let $(A, \triangleright)$ be a left $H$-module coalgebra and
consider $H$ as a right $A$-module coalgebra via the trivial
action, i.e. $h \triangleleft a = \varepsilon_A(a) h$. Then $(A,
H, \triangleleft, \triangleright)$ is a matched pair of Hopf
algebras if and only if $(A, \triangleright)$ is a left $H$-module
algebra and the following compatibility condition holds
\begin{equation}\eqlabel{smash1}
g_{(1)} \otimes g_{(2)} \triangleright a =  g_{(2)} \otimes g_{(1)} \triangleright a
\end{equation}
for all $g \in H$ and $a\in A$. In this case, the associated
bicrossed product $A\bowtie H = A\# H$ is the semi-direct (smash)
product of Hopf algebras as defined by Molnar \cite{Mo}.

2. The fundamental example of a bicrossed product is the Drinfel'd
double $D(H)$. Let $H$ be a finite dimensional Hopf algebra over a
field $k$. Then we have a matched pair of Hopf algebras $(
(H^*)^{\rm cop}, H, \triangleleft, \triangleright)$, where the
actions $\triangleleft$ and $\triangleright$ are defined by:
\begin{equation} \eqlabel{dubluact}
h \triangleleft h^* := \lan h^*, \, S_H^{-1} (h_{(3)}) h_{(1)}
\ran h_{(2)}, \quad h \triangleright h^* := \lan h^*, \, S_H^{-1}
(h_{(2)}) \, ? \, h_{(1)} \ran
\end{equation}
for all $h\in H$ and $h^* \in H^*$ (\cite[Theorem
IX.3.5]{Kassel}). The Drinfel'd double of $H$ is the bicrossed
product associated to this matched pair, i.e. $D(H) = (H^*)^{\rm
cop} \bowtie H$.

3. Let $k$ be a field of characteristic zero, $\mathfrak{g}$,
$\mathfrak{h}$ two Lie algebras and $U(\mathfrak{g})$,
$U(\mathfrak{h})$ the corresponding universal enveloping algebras.
Then there is a bijection between the matched pairs of Lie
algebras $(\mathfrak{g}, \mathfrak{h}, \triangleright,
\triangleleft)$ \cite[Definition 8.3.1]{majid2} and the matched
pairs of Hopf algebras $(U(\mathfrak{g}), U(\mathfrak{h}),
\widehat{\triangleright}, \widehat{\triangleleft})$. The bijection
is given such that there exists a Hopf algebra isomorphism
$U(\mathfrak{g}) \bowtie U(\mathfrak{h}) \cong U(\mathfrak{g}
\bowtie \mathfrak{h})$ (\cite[Proposition 2.4]{Masuoka}).
\end{examples}

A bicrossed product $A \bowtie H$ will be viewed as a left
$A$-module via the restriction of scalars through the canonical
inclusion $i_A : A \to A \bowtie H$, $i_A (a) = a \bowtie 1_H$,
for all $a\in A$. The next result is \cite[Theorem 7.2.3]{majid2}:
Let $A \subseteq E$ be a Hopf subalgebra and $H$ a $A$-complement
of $E$. Then there exists a matched pair of Hopf algebras $(A, H,
\triangleleft, \triangleright)$ such that the multiplication map
$$
m_{E}: A \bowtie H \to E, \qquad m_{E}(a \bowtie h) = ah
$$
for all $a\in A$ and $h\in H$ is an isomorphism of Hopf algebras.
The actions of the matched pair $(A, H, \triangleleft,
\triangleright)$ are constructed as follows for all $a \in A$, $h
\in H$:
\begin{equation}\eqlabel{deduc}
h a = (h_{(1)} \triangleright a_{(1)}) (h_{(2)} \triangleleft
a_{(2)})
\end{equation}

From now on, the matched pair constructed in \equref{deduc} will
be called the \emph{canonical matched pair} associated to the
factorization of $E$ through $A$ and $H$. The following is just
the formal dual of the notion of central map:

\begin{definition}\delabel{lazycocdef}
Let $A$ and $H$ be two Hopf algebras. A coalgebra map $r: H
\rightarrow A$ is called \emph{cocentral} if the following
compatibility holds for all $h\in H$:
\begin{equation}\eqlabel{0aa}
r(h_{(1)}) \ot h_{(2)} = r(h_{(2)}) \ot h_{(1)}
\end{equation}
\end{definition}

The set $CoZ(H, A)$ of all cocentral maps is a group with respect
to the convolution product. We denote by $CoZ^{1}(H, A)$ the
subgroup of $CoZ(H, A)$ of all cocentral maps $r : H \to A$ such
that $r(1_H) = 1_A$. Cocentral maps arise naturally from the
following:

\begin{lemma}\lelabel{lemancoc}
Let $A$ and $H$ be two Hopf algebras. Then there exists a one to
one correspondence between the set of all right $H$-comodule
coalgebra maps $f: H \to A \ot H$ and the set of all cocentral
maps $r: H \to A$. The bijection $(f \mapsto r_f, \, \, r \mapsto
f_r)$ is given as follows:
\begin{equation}\eqlabel{noucor1}
r_f := ( {\rm Id}_A \ot \varepsilon_H) \circ f, \quad f_r (h) =
r(h_{(1)}) \ot h_{(2)}
\end{equation}
for all $h\in H$.
\end{lemma}

\begin{proof}
Straightforward: here $H$ and $A\ot H$ are viewed as right
$H$-comodules via $\Delta_H$.
\end{proof}

For future use we state here the following:

\begin{lemma}\lelabel{bialg-H}
If a Hopf algebra $E$ factorizes into two sub-bialgebras $A$ and
$H$, then $A$ and $H$ are necessarily Hopf algebras.
\end{lemma}

\begin{proof}
If we denote by $S_{E}$ the antipode of $E$, then the antipodes of
$A$ and $H$ are given by the following formulae:
$$
S_A (a) := ({\rm Id}_{A} \ot \epsilon_H) \circ S_E (a \bowtie
1_H), \quad S_H (h) := (\epsilon_A \ot {\rm Id}_{H}) \circ S_E
(1_A \bowtie h)
$$
for all $a \in A$, $h \in H$. Indeed, as $S_{E}$ is the antipode
for $E$ we have:
\begin{eqnarray*}
(a_{(1)} \bowtie 1_{H}) S_{E}(a_{(2)} \bowtie 1_{H}) =
S_{E}(a_{(1)} \bowtie 1_{H}) (a_{(2)} \bowtie 1_{H}) =
\varepsilon_{A}(a) (1_A \bowtie 1_H)
\end{eqnarray*}
for all $a \in A$. By applying the algebra map ${\rm Id}_{A} \ot
\epsilon_H$ to the above identity we obtain that $S_A$ is the
antipode of $A$. The formula for $S_H$ follows in the same manner.
\end{proof}

\section{Classifying complements for Hopf algebras}\selabel{3aplicatii}

Let $A \subseteq E$ be a Hopf subalgebra of $E$ and ${\mathcal F}
(A, E)$ the (possibly empty) isomorphism classes of all
$A$-complements of $E$. The problem of existence of
$A$-complements of $E$ has to be treated 'case by case' for every
given Hopf algebra extension, a computational part of it can not
be avoided. This was the approach used in the similar problem at
the level of groups, i.e. corresponding to the Hopf algebra
extension $k[A] \subseteq k[G]$, for two groups $A$ and $G$ with
$A \leq G$ (see \cite{LPS2} and the references therein). For
example, if $E = k [A_6]$ and $A$ is a proper Hopf subalgebra,
then ${\mathcal F} (A, k [A_6])$ is the empty set. This is based
on the fact that the alternating group $A_6$ has no proper
factorizations \cite{WW}.

In order to answer the (CCP) we need to introduce a few more
concepts:

\begin{definition}\delabel{rigid}
Let $A$ be a Hopf subalgebra of $E$. We define the
\emph{factorization index} of $A$ in $E$ as the cardinal of
${\mathcal F} (A, E)$ and it will be denoted by $[E : A]^f = |\,
{\mathcal F} (A, E) \,|$. The extension $A \subseteq E$ is called
\emph{rigid} if $[E : A]^f = 1$.
\end{definition}

We write down explicitly what a rigid extension of Hopf algebras
$E/A$ means: $[E : A]^f = 1$ if and only if any two
$A$-complements $H$ and $H'$ of $E$ are isomorphic as Hopf
algebras. Equivalently, this can be restated as follows: if $ E
\cong A\bowtie H \cong A \bowtie' H'$ (isomorphism of Hopf
algebras and left $A$-modules), then $H \cong H'$. This is a
Krull-Schmidt-Azumaya type theorem for bicrossed products of Hopf
algebras.

\begin{examples} \exlabel{formegrupuri}
1.  In most cases, for a given extension of Hopf algebras $A
\subseteq E$ the factorization index $[E : A]^f$ is equal to $0$
(i.e. there exists no $A$-complements of $E$) or $1$. For
instance, above we have shown in fact that $[k[A_6] : A]^f = 0$,
for any proper Hopf subalgebra $A$ of the group Hopf algebra
$k[A_6]$.

2. Let $E := A \# H$ be a semidirect product of two Hopf algebras
in the sense of \exref{exempleban}. Then the extension $A \subset
A\# H$ is rigid. Indeed, since $A$ is a normal Hopf subalgebra of
$E$, we obtain that any $A$-complement of $E$ is isomorphic to
$E/A^{+} E$.

3. Examples of extensions $E/A$ for which $[E : A]^f \geq 2$ are
quite rare, which makes them tempting to identify. For instance,
the extension $k[S_{3}] \subseteq k[S_{4}]$ has factorization
index $2$. We shall provide an elaborated way of constructing
examples of Hopf algebra extensions $E/A$ of a given factorization
index in \thref{deform100}.
\end{examples}

\begin{definition}\delabel{amisur}
Let $(A, H, \triangleright, \triangleleft)$ be a matched pair of
Hopf algebras. A unitary cocentral map $r \in CoZ^{1}(H, A)$ is
called a \emph{deformation map} of the matched pair $(A, H,
\triangleright, \triangleleft)$ if the following compatibility
holds for any $g$, $h\in H$:
\begin{equation} \eqlabel{compdef}
r\Bigl( \bigl( h \triangleleft r(g_{(1)})\bigl) \, g_{(2)} \, \Bigl) =
r(h_{(1)}) \, \bigl( h_{(2)} \triangleright r(g)\bigl)
\end{equation}
\end{definition}

Let ${\mathcal D}{\mathcal M} \, (H, A \, | \, (\triangleright,
\triangleleft) ) \subseteq CoZ^{1} (H, A)$ be the set of all
deformation maps of the matched pair $(A, H, \triangleright,
\triangleleft)$. The trivial map $r : H \to A$, $r (h) =
\varepsilon(h) 1_A$ is of course a deformation map. We introduce
the following:

\begin{definition}\delabel{amisur2}
Let $(A, H, \triangleright, \triangleleft)$ be a matched pair of
Hopf algebras. Two deformation maps $r$, $R: H \to A$ are called
\emph{equivalent} and we denote this by $r \sim R$ if there exists
$\sigma: H \to H$ an unitary automorphism of the coalgebra $H$
such that
\begin{equation}\eqlabel{echivamit}
\sigma \bigl( (h \triangleleft r(g_{(1)})) \, g_{(2)} \bigl) = \bigl(
\sigma (h) \triangleleft R (\sigma (g_{(1)}) )  \bigl) \, \sigma (g_{(2)})
\end{equation}
for all $g$, $h\in H$.
\end{definition}

The theorem that gives the answer to the (CCP) for Hopf algebras
is the following:

\begin{theorem}\textbf{(Classification of complements)}\thlabel{clasformelor}
Let $A$ be a Hopf subalgebra of $E$, $H$ an $A$-complement of $E$
and $(A, H, \triangleright, \triangleleft)$ the associated
canonical matched pair. Then:

$(1)$ $\sim$ is an equivalence relation on ${\mathcal D}{\mathcal
M} \, (H, A \, | \, (\triangleright, \triangleleft) )$. We denote by  ${\mathcal
H}{\mathcal A}^{2} (H, A \, | \, (\triangleright, \triangleleft) )$ the quotient set
${\mathcal D}{\mathcal M} \, (H, A \, | \, (\triangleright, \triangleleft) )/ \sim$.

$(2)$ There exists a bijection between the isomorphism classes of
all $A$-complements of $E$ and ${\mathcal H}{\mathcal A}^{2} (H, A
\, | \, (\triangleright, \triangleleft) )$. In particular, the
factorization index of $A$ in $E$ is computed by the formula:
$$
[E : A]^f = | {\mathcal H}{\mathcal A}^{2} (H, A \, | \, (\triangleright,
\triangleleft) )|
$$
\end{theorem}

We prove this theorem in two steps. First, we prove the following
result where a new type of deformation of a given Hopf algebra $H$
is introduced:

\begin{theorem} \textbf{(Deformation of complements)} \thlabel{deformatANM}
Let $A$ be a Hopf subalgebra of $E$, $H$ an $A$-complement of $E$
and $r : H \to A$ a deformation map of the associated canonical
matched pair $(A, H, \triangleright, \triangleleft)$.

$(1)$ Let $f_{r}:H \to A \ot H$ be the coalgebra map defined for
any $h \in H$ by:
$$
f_{r}(h) = r(h_{(1)}) \ot h_{(2)}
$$
Then ${\mathbb H}: = {\rm Im}(f_{r})$ is an $A$-complement of $E
\cong A \bowtie H$.

$(2)$ Let $H_{r} := H$, as a coalgebra, with the new
multiplication $\bullet$ on $H$ defined for any $h$, $g\in H$ as
follows:
\begin{equation}\eqlabel{defoinmult}
h \, \bullet \, g  := \bigl( h \triangleleft r(g_{(1)})\bigl) \,
g_{(2)}
\end{equation}
Then $H_{r} = (H_{r}, \bullet, 1_H, \Delta_H, \varepsilon_H)$ is a
Hopf algebra with the antipode given by
\begin{equation} \eqlabel{antipdef}
S: H_{r} \to H_{r}, \qquad S (h) := S_{H} (h_{(2)}) \triangleleft
(S_{A} \circ r)(h_{(1)})
\end{equation}
for all $h \in H$, called the $r$-deformation of $H$. Furthermore,
$H_r \cong {\mathbb H}$, as Hopf algebras.
\end{theorem}

\begin{proof} $(1)$ Without loss of generality, we can identify
$E = A \bowtie H$, since the multiplication map $m_E: A \bowtie H
\to E$ is a left $A$-linear Hopf algebra isomorphism. It follows
from \leref{lemancoc} that $f_r: H \to A \ot H$ is a unit
preserving injective coalgebra map. Thus, ${\mathbb H} = {\rm
Im}(f_{r})$ is a subcoalgebra of $E = A \bowtie H$. We will denote
by $\widetilde{f_{r}} : H \to {\mathbb H}$ the coalgebra
isomorphism induced by $f_{r}$. We shall prove that ${\mathbb H}$
is a sub-bialgebra of $E = A \bowtie H$ and moreover, $E$
factorizes trough $A$ and ${\mathbb H}$. Indeed, using
\equref{compdef} it follows that ${\mathbb H}$ is also a
subalgebra of $E$ since for any $h$, $g \in H$ we have:
\begin{eqnarray*}
\bigl(r(h_{(1)}) \bowtie h_{(2)}\bigl)\bigl(r(g_{(1)}) \bowtie
g_{(2)}\bigl) &=& \underline{r(h_{(1)}) \bigl(h_{(2)} \rhd
r(g_{(1)})\bigl)} \bowtie \bigl(h_{(3)} \lhd r(g_{(2)})\bigl)
g_{(3)}\\
&\stackrel{\equref{compdef}} {=}&  r \Bigl(\bigl(h_{(1)} \lhd
r(g_{(1)})\bigl) g_{(2)}\Bigl) \bowtie \bigl(h_{(2)} \lhd
r(g_{(3)})\bigl) g_{(4)} \\
&\stackrel{\equref{0aa}} {=}& f_r \bigl( (h \triangleleft
r(g_{(1)}) g_{(2)} \bigl) \in {\mathbb H}
\end{eqnarray*}
Therefore, ${\mathbb H}$ is a sub-bialgebra of $E$. Consider now
the left $A$-linearization of $f_r$, i.e. the map $\varphi: A \ot
H \to A \ot H$ given by $ \varphi(a \ot h) = a f_{r}(h)$, for all
$a \in A$ and $h \in H$. Then $\varphi$ is a bijection with the
inverse given by: $\varphi^{-1} \bigl(a \otimes h) = a S_{A}
\bigl(r(h_{(1)})\bigl) \ot h_{(2)}$. Since $\varphi$ decomposes as
$\varphi = m \circ ({\rm Id}_{A} \ot \widetilde{f_{r}})$, where
$m: A \ot {\mathbb H} \to E$ is the multiplication map, it follows
that $m = \varphi \circ ({\rm Id}_{A} \ot
\widetilde{f_{r}}^{-1})$. Thus the multiplication map $m: A \ot
{\mathbb H} \to E$ is bijective and hence $E$ factorizes through
$A$ and ${\mathbb H}$. Moreover, ${\mathbb H}$ is necessarily a
Hopf subalgebra of $E$ by \leref{bialg-H} and hence ${\mathbb H}$
is an $A$-complement of $E = A \bowtie H$, as needed.

$(2)$ $\widetilde{f_{r}} : H \to {\mathbb H}$ is a unit preserving
coalgebra isomorphism. Moreover, in the proof of part $(1)$ we
obtained that $\widetilde{f_{r}} (h) \widetilde{f_{r}} (h)  =
\widetilde{f_{r}} (h \bullet g)$, where $\bullet$ is the
multiplication on $H$ given by \equref{defoinmult}. Therefore,
$\widetilde{f_{r}} : H_{r} \to {\mathbb H}$ is a bialgebra
isomorphism between the Hopf algebra ${\mathbb H}$ and $H_r$.
Thus, $H_r$ is a Hopf algebra with the antipode given by
\equref{antipdef}.
\end{proof}

\begin{remarks} \relabel{cudeformar}
1. Assume that in \thref{deformatANM} the unitary cocentral map
$r: H \to A$ is the trivial one $r (h) = \varepsilon_H (h) 1_A$ or
the right action $\triangleleft : H\ot A \to H $ of $A$ on $H$ is
the trivial action, i.e. $h \triangleleft a = \varepsilon_A (a)
h$, for all $h\in H$ and $a\in A$. Then $H_{r} = H$ as Hopf
algebras. In general, the new Hopf algebra $H_{r}$ may not be
isomorphic to $H$ as a Hopf algebra: an example will be provided
in \thref{findmare}.

2. At this point we should notice that there are two other
deformations of a given Hopf algebra in the literature. The first
one was introduced by Drinfel'd \cite{dr}: the comultiplication of
a Hopf algebra $H$ is deformed using an invertible element $R\in
H\ot H$, called twist, in order to obtain a new Hopf algebra
$H^R$. The dual case was defined by Doi \cite{Doi}: the algebra
structure of a Hopf algebra $H$ was deformed using a Sweedler
cocycle $\tau : H \ot H \to k$ as follows: let $H_{\tau} = H$, as
a coalgebra, with the new multiplication given by
$$
h \cdot g := \tau (h_{(1)}, g_{(1)}) \, h_{(2)} g_{(2)} \,
\tau^{-1} (h_{(3)}, g_{(3)})
$$
for all $h$, $g\in H$. Then $H_{\tau}$ is a new Hopf algebra
\cite[Theorem 1.6]{Doi} and among several applications we mention
that the Drinfel'd double $D(H)$ is a special case of this
deformation \cite{DT2}.
\end{remarks}

Next we shall prove the converse of \thref{deformatANM}.

\begin{theorem}\textbf{(Description of complements)}\thlabel{descformelor}
Let $A$ be a Hopf subalgebra of $E$, $H$ an $A$-complement of $E$
with the associated canonical matched pair $(A, H, \triangleright,
\triangleleft)$ and let ${\mathbb H}$ be an arbitrary
$A$-complement of $E$. Then there exists an isomorphism of Hopf
algebras $ {\mathbb H} \cong H_{r}$, for some deformation map $r :
H \to A$ of the matched pair $(A, H, \triangleright,
\triangleleft)$.
\end{theorem}

\begin{proof} The multiplication map $m_E: A \bowtie H \to E$ is a left
$A$-linear Hopf algebra isomorphism and we denote its inverse by
$\nu$. In fact, without loss of generality, we can identify $E = A
\bowtie H$. We denote the multiplication map associated to the
$A$-complement ${\mathbb H}$ by $m'_E : A \ot {\mathbb H} \to E$
and its inverse by $\mu$. Define $f: H \to {\mathbb H}$ as the
composition:
\begin{equation}\eqlabel{f-ul}
f: \, H \, \stackrel{i}{\hookrightarrow} \, E \,
\stackrel{\mu}{\longrightarrow} \, A \ot {\mathbb H} \,
\stackrel{\varepsilon_A \ot {\rm Id}} {\longrightarrow} \,
{\mathbb H}
\end{equation}
Then, $f$ is a unitary coalgebra isomorphism with the inverse
$f^{-1}$ given by the composition:
\begin{equation}\eqlabel{f-ulinv}
f^{-1}: \, {\mathbb H} \, \stackrel{i}{\hookrightarrow} \, E \,
\stackrel{\nu}{\longrightarrow} \, A \ot H \,
\stackrel{\varepsilon_A \ot {\rm Id}} {\longrightarrow} \, H
\end{equation}
The proof will be finished if we construct a deformation map $r :
H \to A$ of the canonical matched pair $(A, H, \triangleright,
\triangleleft)$ such that $f: H_r \to {\mathbb H}$ is an algebra
map. This deformation map $r : H \to A$ is given by the
composition of the following maps:
\begin{equation}\eqlabel{r-ul}
r: \, H \, \stackrel{f}{\longrightarrow} \, {\mathbb H}\,
\stackrel{i}{\hookrightarrow} \, E \,
\stackrel{\nu}{\longrightarrow} \, A \ot H \, \stackrel{{\rm Id}
\ot \varepsilon_H} {\longrightarrow} \, A
\end{equation}
We shall prove this. We view ${\mathbb H}$ as a right $H$-comodule
along the coalgebra map $ f^{-1}: {\mathbb H} \to H$. Then, $f : H
\to {\mathbb H}$ is right $H$-colinear, since its inverse $f^{-1}$
is. We denote by $\tilde{f} : H \to A \ot H$ the following
composition:
$$
\tilde{f} : \, H \, \stackrel{f} {\longrightarrow} \, {\mathbb H}
\, \stackrel{i}{\hookrightarrow} \, E \, {\cong} \, A \bowtie H
$$
Then $\tilde{f} : H \to A \ot H$ is a unit-preserving right
$H$-colinear and coalgebra map. It follows from \leref{lemancoc}
that there exists a unique unit-preserving cocentral map $r: H \to
A$ such that $\tilde{f} (h) = r(h_{(1)}) \ot h_{(2)}$, for all
$h\in H$ and moreover the map $r$ is given by \equref{r-ul}. Now,
$ {\rm Im} ( \tilde{f} ) = {\rm Im} (f) = {\mathbb H} $ is a Hopf
subalgebra of $E = A \bowtie H$, since ${\mathbb H}$ is an
$A$-complement. Thus, for any $h$, $g \in H$ we have that $
\tilde{f} (h) \tilde{f} (g) \in {\rm Im} (\tilde{f})$. Now, we
have:
$$
\tilde{f} (h) \tilde{f} (g) = r (h_{(1)}) \bigl( h_{(2)}
\triangleright r (g_{(1)}) \bigl) \, \bowtie \, \bigl( h_{(3)}
\triangleleft  r (g_{(2)}) \bigl) g_{(3)}
$$
This element is of the form $ \tilde{f} (x) = r(x_{(1)}) \ot
x_{(2)}$, for some $x\in H$ if and only if $x = \bigl( h
\triangleleft  r (g_{(1)}) \bigl) g_{(2)}$ and $r$ is a descent
map of the canonical matched pair $(A, H, \triangleright,
\triangleleft)$. Indeed, if we apply $\varepsilon_A \ot {\rm
Id}_H$ to the identity $\tilde{f} (h) \tilde{f} (g) = \tilde{f}
(x)$, we obtain the above formula for $x$ while by applying ${\rm
Id}_A \ot \varepsilon_H$ to the formula $\tilde{f} (h) \tilde{f}
(g) = \tilde{f} (\bigl( h \triangleleft  r (g_{(1)}) \bigl)
g_{(2)})$ we obtain that $r: H \to A$ is a deformation map.
Furthermore, in this case we have $f (h) f (g) = f \bigl ( \bigl(
h \triangleleft r (g_{(1)}) \bigl) g_{(2)}  \bigl) = f (h \bullet
g)$, that is $f : H_r \to {\mathbb H}$ is an algebra map, as
needed.
\end{proof}

We are now ready to prove \thref{clasformelor}:

\begin{proof}[The proof of \thref{clasformelor}]
It follows from \thref{descformelor} that in order to classify all
$A$-complements of $E$ we can consider only $r$-deformations of
$H$, for various deformation maps $r : H \to A$. Let $r$, $R : H
\to A$ be two deformation maps. As the coalgebra structure on
$H_r$ and $H_R$ coincide with the one of $H$, we obtain that the
Hopf algebras $H_r$ and $H_R$ are isomorphic if and only if there
exists $\sigma : H \to H$ a unitary coalgebra isomorphism such
that $\sigma : H_r \to H_R$ is also an algebra map. Taking into
account the definition of the multiplication on $H_r$ given by
\equref{defoinmult} we obtain that $\sigma$ is an algebra map if
and only if the compatibility condition \equref{echivamit} of
\deref{amisur2} holds, i.e. $r \sim R$. Hence, $r \sim R$ if and
only if $\sigma: H_r \to H_R$ is an isomorphism of Hopf algebras.
Thus we obtain that $\sim$ is an equivalence relation on
${\mathcal D}{\mathcal M} (H, A \, | \, (\triangleright,
\triangleleft) )$ and the map
$$
{\mathcal H}{\mathcal A}^{2} (H, A \, | \, (\triangleright,
\triangleleft) ) \to {\mathcal F} (A, E),  \qquad \overline{r}
\mapsto H_{r}
$$
where $\overline{r}$ is the equivalence class of $r$ via the
relation $\sim$, is well defined and a bijection between sets.
This finishes the proof.
\end{proof}

\section{Examples}\selabel{secexemple}
In this section we shall provide an example of a Hopf algebra
extension $A \subseteq E$ whose factorization index is arbitrary
large. For a positive integer $n$ we denote by $U_n (k) = \{
\omega \in k \, | \, \omega^n = 1 \}$ the cyclic group of $n$-th
roots of unity in $k$ and by $\nu (n) = |U_n (k)|$ the order of
$U_n (k)$. $C_n$ will be the cyclic group of order $n$ generated
by $c$ or $d$ (if we consider two copies of $C_n$) and $k$ will be
a field of characteristic $\neq 2$. Let $A := H_{4}$ be the
Sweedler's $4$-dimensional Hopf algebra generated by $g$ and $x$
subject to the relations:
$$
g^{2} = 1, \quad x^{2} = 0, \quad x g = -g x
$$
with the coalgebra structure given such that $g$ is a group-like
element and $x$ is $(1, g)$-primitive. \cite[Proposition
4.3]{abm1} proves that there exists a bijective correspondence
between the set of all matched pairs $(H_4, k[C_n], \triangleleft,
\triangleright)$ and the group $U_n (k)$ such that the matched
pair $(H_4, k[C_n], \triangleleft, \triangleright)$ associated to
$\omega \in U_n (k)$ is given as follows: $\triangleleft: k[C_{n}]
\ot H_{4} \to k[C_{n}]$ is the trivial action and $\triangleright:
k[C_{n}] \ot H_{4} \to H_{4}$ is defined by:
\begin{equation}\eqlabel{bicrossedH4n}
c^{i} \triangleright g = g, \quad c^{i} \triangleright x = \omega^{i} x, \quad c^{i}
\triangleright gx = \omega^{i} \, gx
\end{equation}
for all $i = 0, 1, \cdots, n-1$. We denote by $H_{4n, \, \omega}$
the bicrossed product $H_4 \bowtie k[C_n]$ associated to this
matched pair: $H_{4n, \, \omega}$ is the $4n$-dimensional quantum
group generated by $g$, $x$ and $c$ subject to the relations:
$$
g^{2} = c^n = 1, \quad x^{2} = 0, \quad x g = -g x, \quad c g = g
c, \quad c x = \omega \, x c
$$
with the coalgebra structure given such that $g$ and $c$ are
group-like elements and $x$ is $(1, g)$-primitive. A $k$-basis in
$H_{4n, \, \omega}$ is given by $\{ c^i, \, g c^i, \, x c^i, g
xc^i \, | \, i = 0, \cdots, n-1 \}$.

Let $\xi $ be a generator of the group $U_n(k)$. In what follows
we will construct a family of matched pairs of Hopf algebras
$(k[C_{n}], \, H_{4n, \, \xi^{t}}, \, \triangleleft^{l} ,\,
\triangleright)$ such that the Hopf algebra $H_{4n, \, \xi^{t -
lp}}$ will appear as an $r$-deformation of $H_{4n, \, \xi^{t}}$.

\begin{theorem}\thlabel{deform100}
Let $k$ be a field of characteristic $\neq 2$, $n$ a positive
integer, $\xi$ a generator of $U_n(k)$, $t \in \{0, 1, ..., \nu(n)
- 1\}$ and $C_n = \lan d ~|~ d^{n} = 1 \ran $ the cyclic group of
order $n$. Then:

$(1)$ For any $l \in \{0, 1, \cdots, \nu(n) - 1\}$ there exists a
matched pair $(k[C_{n}], \, H_{4n, \,  \xi^{t}}, \,
\triangleleft^{l} ,\, \triangleright)$, where $\triangleright :
H_{4n, \,  \xi^{t}} \, \ot \, k[C_n] \to k[C_n]$ is the trivial
action and the right action $\triangleleft^{l}: H_{4n, \, \xi^{t}}
\, \ot \, k[C_n]  \to H_{4n, \,  \xi^{t}}$ is given for any $i$,
$k = 0, 1, ..., n-1$ by:
\begin{equation}\eqlabel{mpair100}
c^i \triangleleft^l d^k = c^i, \quad (gc^i) \triangleleft^l d^k = gc^i, \quad (xc^i)
\triangleleft^l d^k = \xi^{lk} \, x c^i, \quad  (gxc^i)\triangleleft^l d^k =
\xi^{lk} \, gx c^i
\end{equation}

$(2)$ The deformation maps associated to the matched pair
$(k[C_{n}], \, H_{4n, \,  \xi^{t}}, \, \triangleleft^{l} ,\,
\triangleright)$ are the algebra maps defined for any $p \in \{0,
1, \cdots, n - 1\}$ as follows:
$$
r_{p}: H_{4n, \, \xi^{t}} \to k[C_{n}], \quad r_{p}(g) = 1, \quad
r_{p}(c) = d^{p}, \quad r_{p}(x) = 0
$$
Furthermore, the $r_p$-deformation of $H_{4n, \, \xi^{t}}$ is
$H_{4n, \, \xi^{t - lp}}$, i.e. $(H_{4n, \, \xi^{t}})_{r_{p}} =
H_{4n, \, \xi^{t - lp}}$.
\end{theorem}

\begin{proof}
$(1)$ The compatibility condition \equref{mp4} is trivially
fulfilled since $\triangleright$ is the trivial action and
$k[C_{n}]$ is cocommutative. Moreover, \equref{mp3} becomes:
\begin{equation}\eqlabel{defo1}
(yz) \triangleleft^{l} a = (y \triangleleft^{l} a_{(1)})(z \triangleleft^{l} a_{(2)})
\end{equation}
for all $y$, $z \in H_{4n, \,  \xi^{t}}$ and $a \in
k[C_{n}]$. Since we have:
$$
c^{i} \triangleleft^{l} d^{k} = c^{i}, \qquad g \triangleleft^{l} d^{k} = g, \qquad
x \triangleleft^{l} d^{k} = \xi^{lk}x
$$
then it is straightforward to see that \equref{defo1} indeed
holds. The fact that $\triangleleft^{l}: H_{4n, \,  \xi^{t}} \ot
k[C_{n}] \to H_{4n, \,  \xi^{t}}$ is a coalgebra map is just a
routinely computation. Finally, we only need to prove that the
action $\triangleleft^{l}$ respects the relations in $k[C_{n}]$,
respectively $H_{4n, \, \xi^{t}}$. For instance, we have:
$$
xc^{i} \triangleleft^{l} d^{n} = (xc^{i} \triangleleft^{l} d^{n-1}) \triangleleft^{l} d =
\xi^{l(n-1)} xc^{i} \triangleleft^{l} d = \xi^{ln} xc^{i} = xc^{i}
$$
$$
xg \triangleleft^{l} d^{k} = (x \triangleleft^{l} d^{k})(g \triangleleft^{l} d^{k}) =
\xi^{lk}xg = - \xi^{lk}gx = -gx \triangleleft^{l} d^{k}
$$
$$
cx \triangleleft^{l} d^{k} = (c \triangleleft^{l} d^{k})(x \triangleleft^{l} d^{k}) =
\xi^{lk} cx = \xi^{lk} \xi^{t} xc = \xi^{t} xc \triangleleft^{l} d^{k}
$$
Proving that the rest of the compatibilities also hold is a
routinely check.

$(2)$ Let $r: H_{4n, \, \xi^{t}} \to k[C_{n}]$ be a deformation
map. By applying \equref{0aa} for $xc^{i}$ and $gxc^{i}$, where $i
= 0, 1, \cdots , n-1$ we obtain:
$$
r(c^{i}) \ot xc^{i} + r(xc^{i}) \ot gc^{i} = r(xc^{i}) \ot c^{i} +
r(gc^{i}) \ot xc^{i}
$$
$$
r(gc^{i}) \ot gxc^{i} + r(gxc^{i}) \ot c^{i} = r(gxc^{i}) \ot
gc^{i} + r(c^{i}) \ot gxc^{i}
$$
Hence, it follows that $r(xc^{i}) = r(gxc^{i}) = 0 $ and $r(c^{i})
= r(gc^{i})$ for all $i \in 0, 1, ..., n-1$. In particular we have
$r(g) = 1$ and $r(x) = 0$. Moreover, since $r$ is also a coalgebra
map then $r(c)$ is a grouplike element from $k[C_{n}]$. Consider
$r(c) = d^{p}$, for some $p = 0, 1, \cdots, n-1$. For the rest of
the proof we will denote this map by $r_{p}$. As $\triangleright$
is the trivial action then the compatibility condition
\equref{compdef} simplifies to:
\begin{equation}\eqlabel{compdef20}
r_{p}\Bigl( \bigl( y \triangleleft^{l} r(z_{(1)})\bigl) \, z_{(2)} \,
\Bigl) = r_{p}(y) r_{p}(z)
\end{equation}
for all $y$, $z \in H_{4n, \,\xi^{t}}$. By applying
\equref{compdef20} for $c^{i}$ and $c^{j}$, where $i$, $j \in 0,
1, ..., n-1$ we get $r_{p}(c^{i+j}) = r_{p}(c^{i}) r_{p}(c^{j})$.
Hence, we have
\begin{equation}\eqlabel{defrp}
r_{p}(c^{i}) = r_{p}(gc^{i}) = d^{ip}
\end{equation}
for all $i = 0, 1, \cdots, n-1$. Now by using \equref{defrp} and
the fact that $r_{p}(xc^{i}) = r_{p}(gxc^{i}) = 0 $, for any $i =
0, 1, \cdots , n-1$ we can easily prove that $r_{p}: H_{4n, \,
\xi^{t}} \to k[C_{n}]$ is an algebra map. Finally, we are left to
prove that \equref{compdef20} holds. For instance we have:
\begin{eqnarray*}
r_{p}\Bigl( \bigl( gc^{i} \triangleleft^{l} r_{p}(c^{j})\bigl) \,
c^{j} \, \Bigl) &=& r_{p}\Bigl( \bigl( gc^{i} \triangleleft^{l}
d^{pj}\bigl) \, c^{j} \, \Bigl) = r_{p}(g c^{i+j}) = r_{p}(gc^{i})
r_{p}(c^{j})\\
r_{p}\Bigl( \bigl( xc^{i} \triangleleft^{l} r_{p}(c^{j})\bigl) \,
c^{j} \, \Bigl) &=& r_{p}\Bigl( \bigl( xc^{i} \triangleleft^{l}
d^{pj}\bigl) \, c^{j} \, \Bigl) = r_{p}(\xi ^{lpj} x c^{i+j}) = 0
= r_{p}(xc^{i}) r_{p}(c^{j})\\
r_{p}\Bigl( \bigl( y \triangleleft^{l} r_{p}((xc^{i})_{(1)})\bigl)
\, (xc^{i})_{(2)} \, \Bigl) &=& r_{p}\Bigl( \bigl( y
\triangleleft^{l} r_{p}(xc^{i})\bigl) \, c^{i} \, \Bigl) +
r_{p}\Bigl( \bigl( y \triangleleft^{l} r_{p}(gc^{i})\bigl) \,
xc^{i} \, \Bigl) \\
&=& 0 = r_{p}(y) r_{p}(xc^{i})
\end{eqnarray*}
for all $i$, $j = 0, 1, \cdots, n-1$ and $y \in H_{4n, \,
\xi^{t}}$. By a straightforward computation it can be seen that
\equref{compdef20} also holds for the remaining elements of the
$k$-basis of $H_{4n, \, \xi^{t}}$.

Now, the algebra structure of $(H_{4n, \, \xi^{t}})_{r_{p}}$ is
given by \equref{defoinmult}. Thus, in $(H_{4n, \,
\xi^{t}})_{r_{p}}$ we have:
\begin{eqnarray*}
g \bullet g &=& \bigl(g \triangleleft^{l} r_{p}(g)\bigl)g = (g
\triangleleft^{l} 1) g = g^{2} = 1\\
x \bullet x &=& (x \triangleleft^{l} r(x)) + (x \triangleleft^{l}
r(g)) x = x^{2} = 0\\
c^{n-1} \bullet c &=& \bigl(c^{n-1} \triangleleft^{l}
r_{p}(c)\bigl)c = \bigl(c^{n-1} \triangleleft^{l} c^{p}\bigl)c =
c^{n-1}c = c^{n} = 1\\
g \bullet x &=& (g \triangleleft^{l} r(x)) + (g \triangleleft^{l}
r(g)) x = gx = - xg = -(x \triangleleft^{l} r(g)) g = - x \bullet g\\
c \bullet x &=& (c \triangleleft^{l} r(x)) + (c \triangleleft^{l}
r(g)) x = cx = \xi^{t} xc = \xi^{t-lp} (\xi^{lp}xc) \\
&=& \xi^{t-lp} (x \triangleleft^{l} r_{p}(c))c = \xi^{t-lp} x
\bullet c
\end{eqnarray*}
This shows that $(H_{4n, \, \xi^{t}})_{r_{p}} = H_{4n, \, \xi^{t -
lp}}$.
\end{proof}

The bicrossed product $k[C_n] \bowtie^l H_{4n, \, \xi^{t}}$
associated to the matched pair from \thref{deform100} is the Hopf
algebra generated by $g$, $x$, $c$ and $d$ subject to the
relations:
$$
g^{2} = c^n = d^n = 1, \quad x^{2} = 0, \quad c g = g c, \quad cd
= dc, \quad gd = dg,
$$
$$
x g = -g x, \quad c x = \xi^{t} \, x c, \quad xd = \xi^{l} \, dx
$$
with the coalgebra structure given such that $g$, $c$, $d$ are
group-like elements and $x$ is a $(1, g)$-primitive element. We
denote by $H_{4n^2, \, \xi, \, t, \, l}$ this family of quantum
groups, for any $l$, $t \in \{0, 1, ..., \nu(n) - 1\}$ and $\xi$ a
generator of order $\nu(n)$ of the group $U_n(k)$. In what follows
we view $H_{4n^2, \, \xi, \, t, \, l}$ as a Hopf algebra extension
of the group algebra $k[C_n] = k\lan d \, | \, d^n = 1 \ran$. In
this context, $H_{4n, \, \xi^{t}}$ is a $k[C_n]$-complement of
$H_{4n^2, \, \xi, \, t, \, l}$.

Before stating the main result of this section we recall from
\cite[Theorem 4.10]{abm1} the number of types of isomorphisms of
Hopf algebras $H_{4n, \, \omega}$, where $\omega \in U_n(k)$; we
denote this number by $\# H_{4n, \, \omega}$. If $\nu(n) =
p_1^{\alpha_1} \cdots p_r^{\alpha_r}$ is the prime decomposition
of $\nu(n) = |U_n(k)|$ then we have:
\begin{equation}\eqlabel{izonumberh4}
\# H_{4n, \, \omega} = \left\{
\begin{array}{rcl} (\alpha_1 + 1)(\alpha_2 + 1) \cdots (\alpha_r + 1), & \mbox{if}& \nu(n) \hspace{2mm} \mbox{is odd}\\
\alpha_1(\alpha_2 + 1) \cdots (\alpha_r + 1), & \mbox{if}&
\nu(n)\hspace{2mm} \mbox{is even and} \hspace{2mm} p_{1}=2
 \end{array}\right.
 \end{equation}

The main result of this section now follows: it computes the
factorization index of the extension $k[C_{n}] \subseteq H_{4n^2,
\, \xi, \, t, \, 1}$.

\begin{theorem}\thlabel{findmare}
Let $k$ be a field of characteristic $\neq 2$, $n$ a positive
integer, $\xi$ a generator of $U_n(k)$ and $(k[C_{n}], \, H_{4n,
\, \xi^{\nu(n) - 1}}, \, \triangleleft^{1} ,\, \triangleright)$
the matched pair where $\triangleright$ is the trivial action and
$\triangleleft^{1}$ is given by \equref{mpair100} for $l=1$. Then:

$1)$ $(H_{4n, \, \xi^{\nu(n)-1}})_{r_{p}} = H_{4n, \, \xi^{\nu(n)
- 1 - p}}$, for all $p = 0, 1, \cdots , \nu(n) - 1$. Thus, any
$H_{4n, \, \xi^{p}}$ appears as a deformation of $H_{4n, \,
\xi^{\nu(n)-1}}$, for some deformation map $r_{p}$.

$2)$ Assume that $\nu(n)$ is odd and $\nu(n) = p_1^{\alpha_1}
\cdots p_r^{\alpha_r}$ is the prime decomposition of $\nu(n)$.
Then we have $(\alpha_{1}+1)(\alpha_{2}+1) ... (\alpha_{r} + 1)$
non-isomorphic deformations of $H_{4n, \, \xi^{\nu(n)-1}}$ and
thus $[ H_{4n^2, \, \xi, \, t, \, 1} \, : \, k[C_n] \, ]^f =
(\alpha_{1}+1)(\alpha_{2}+1) ... (\alpha_{r} + 1)$.

$3)$ Assume that $\nu(n)$ is even and $\nu(n) = 2^{\alpha_1}
p_2^{\alpha_2} \cdots p_r^{\alpha_r}$ is the prime decomposition
of $\nu(n)$. Then we have $\alpha_{1}(\alpha_{2}+1) ...
(\alpha_{r} + 1)$ non-isomorphic deformations of $H_{4n, \,
\xi^{\nu(n)-1}}$ and thus $[ H_{4n^2, \, \xi, \, t, \, 1} \, : \,
k[C_n] \, ]^f = \alpha_{1}(\alpha_{2}+1) ... (\alpha_{r} + 1)$.
\end{theorem}

\begin{proof}
$1)$ It follows by applying \thref{deform100} for $l = 1$ and $t =
\nu(n) - 1$. As any $H_{4n, \, \xi^{p}}$ appears as a deformation
of $H_{4n, \, \xi^{\nu(n)-1}}$ via some deformation map $r_{p}$,
the last two statements are just easy consequences of
\equref{izonumberh4}.
\end{proof}

\section{Classifying complements for Lie algebras}\selabel{deflie}
Let $\mathfrak{g} \subseteq \Xi$ be a Lie subalgebra of $\Xi$. A
Lie subalgebra $\mathfrak{h}$ of $\Xi$ is called a
\emph{complement} of $\mathfrak{g}$ in $\Xi$ (or a
\emph{$\mathfrak{g}$-complement} of $\Xi$) if $\Xi = \mathfrak{g}
+ \mathfrak{h}$ and $\mathfrak{g} \cap \mathfrak{h} = \{0\}$. In
this case we say that the Lie algebra $\Xi$ factorizes through
$\mathfrak{g}$ and $\mathfrak{h}$. Related to these concepts, the
bicrossed product associated to a matched pair of Lie algebras was
introduced in \cite{majid}. We collect here some basic facts: for
more details we refer the reader to \cite{majid}, \cite[Chapter
8]{majid2} or \cite{Masuoka}. A \emph{matched pair} of Lie
algebras is a quadruple $(\mathfrak{g}, \mathfrak{h},
\triangleright, \triangleleft)$, where $\mathfrak{g}$,
$\mathfrak{h}$ are Lie algebras, $\mathfrak{g}$ is a left
$\mathfrak{h}$-Lie module under $\triangleright: \mathfrak{h} \ot
\mathfrak{g} \to \mathfrak{g}$, $\mathfrak{h}$ is a right
$\mathfrak{g}$-Lie module under $\triangleleft: \mathfrak{h} \ot
\mathfrak{g} \to \mathfrak{h}$ and the following compatibilities
hold for all $a$, $b \in \mathfrak{g}$, $x$, $y \in \mathfrak{h}$:
\begin{equation}\eqlabel{mpLie1}
x \triangleright [a, \, b] = [x \triangleright a, \,  b] + [a, \,
x \triangleright b] + (x \triangleleft a) \triangleright b - (x
\triangleleft b) \triangleright a
 \end{equation}
\begin{equation}\eqlabel{mpLie2}
[x, \, y] \triangleleft a = [x,\, y \triangleleft a] + [x
\triangleleft a, \, y] + x \triangleleft (y \triangleright a) - y
\triangleleft (x \triangleright a)
\end{equation}
The fact that $\mathfrak{g}$ is a left $\mathfrak{h}$-Lie module
under $\triangleright: \mathfrak{h} \ot \mathfrak{g} \to
\mathfrak{g}$ and $\mathfrak{h}$ is a right $\mathfrak{g}$-Lie
module under $\triangleleft: \mathfrak{h} \ot \mathfrak{g} \to
\mathfrak{h}$ can be written explicitly as follows:
\begin{equation}\eqlabel{mpLie3}
[x, \, y] \triangleright a = x \triangleright (y \triangleright a)
- y \triangleright (x \triangleright a), \qquad x \triangleleft
[a, \, b] = (x \triangleleft a) \triangleleft b - (x \triangleleft
b) \triangleleft a
\end{equation}

The following is \cite[Proposition 8.3.2]{majid2}: If
$(\mathfrak{g}, \mathfrak{h}, \triangleright, \triangleleft)$ is a
matched pair of Lie algebras then the direct sum $\mathfrak{g}
\oplus \mathfrak{h}$ together with the bracket defined by:
\begin{equation}\eqlabel{bracketmpLie}
[a \oplus x, \, b \oplus y] = \bigl([a, \, b] + x \triangleright b
- y \triangleright a \bigl) \oplus \bigl([x, \, y] +
x\triangleleft b - y \triangleleft a \bigl)
\end{equation}
for all $a$, $b \in \mathfrak{g}$, $x$, $y \in\mathfrak{h}$ is a
Lie algebra, called the \emph{bicrossed product} of $\mathfrak{g}$
and $\mathfrak{h}$, and will be denoted by $\mathfrak{g} \bowtie
\mathfrak{h}$. The Lie algebra $\mathfrak{h} \cong \{0\}\oplus
\mathfrak{h} $ is a complement of $\mathfrak{g} \cong \mathfrak{g}
\oplus \{0\}$ in the bicrossed product $\mathfrak{g} \bowtie
\mathfrak{h}$. Conversely, if $\mathfrak{h}$ is a
$\mathfrak{g}$-complement of $\Xi$, then there exists a matched
pair of Lie algebras $(\mathfrak{g}, \mathfrak{h}, \triangleright,
\triangleleft)$ such that the corresponding bicrossed product
$\mathfrak{g} \bowtie \mathfrak{h}$ is isomorphic as a Lie algebra
with $\Xi$. The actions of the matched pair $(\mathfrak{g},
\mathfrak{h}, \triangleright, \triangleleft)$ arises from the
unique decomposition:
\begin{equation}\eqlabel{Lie456}
[x, \, a] = x \triangleright a \oplus x \triangleleft a
\end{equation}
for all $a \in \mathfrak{g}$, $x \in \mathfrak{h}$. The matched
pair constructed in \equref{Lie456} will be called the
\emph{canonical matched pair} associated to the
$\mathfrak{g}$-complement $\mathfrak{h}$ of $\Xi$.

For a Lie subalgebra $\mathfrak{g}$ of $\Xi$ we denote by
${\mathcal F} (\mathfrak{g}, \, \Xi)$ the isomorphism classes of
$\mathfrak{g}$-complements of $\Xi$. The \emph{factorization
index} of $\mathfrak{g}$ in $\Xi$ is defined as $[\Xi :
\mathfrak{g}]^f := |\, {\mathcal F} (\mathfrak{g}, \, \Xi) \,|$.

\begin{definition} \delabel{deformaplie}
Let $(\mathfrak{g}, \mathfrak{h}, \triangleright, \triangleleft)$
be a matched pair of Lie algebras. A $k$-linear map $r:
\mathfrak{h} \to \mathfrak{g}$ is called a \emph{deformation map}
of the matched pair $(\mathfrak{g}, \mathfrak{h}, \triangleright,
\triangleleft)$ if the following compatibility holds for any $x$,
$y \in \mathfrak{h}$:
\begin{equation}\eqlabel{factLie}
r\bigl([x, \,y]\bigl) \, - \, \bigl[r(x), \, r(y)\bigl] = r \bigl(
\, y \triangleleft r(x) - x \triangleleft r(y) \, \bigl) + x
\triangleright r(y) - y \triangleright r(x)
\end{equation}
\end{definition}

We denote by ${\mathcal D}{\mathcal M} \, (\mathfrak{h},
\mathfrak{g} \, | \, (\triangleright, \triangleleft) )$ the set of
all deformation maps of the matched pair $(\mathfrak{g},
\mathfrak{h}, \triangleright, \triangleleft)$. The right hand side
of \equref{factLie} measures how far a deformation map is from
being a Lie algebra map. Using this concept the following
deformation of a given Lie algebra is proposed:

\begin{theorem}\thlabel{deforLie}
Let $\mathfrak{g}$ be a Lie subalgebra of $\Xi$, $\mathfrak{h}$ a
given $\mathfrak{g}$-complement of $\Xi$ and $r: \mathfrak{h} \to
\mathfrak{g}$ a deformation map of the associated canonical
matched pair $(\mathfrak{g}, \mathfrak{h}, \triangleright,
\triangleleft)$.

$(1)$ Let $f_{r}: \mathfrak{h} \to \Xi = \mathfrak{g} \oplus
\mathfrak{h}$ be the $k$-linear map defined for any $x \in
\mathfrak{h}$ by:
$$f_{r}(x) = r(x) \oplus x$$
Then $\widetilde{\mathfrak{h}} : = {\rm Im}(f_{r})$ is a
$\mathfrak{g}$-complement of $\Xi$.

$(2)$ $\mathfrak{h}_{r} := \mathfrak{h}$, as a $k$-module, with
the new bracket defined for any $x$, $y \in \mathfrak{h}$ by:
\begin{equation}\eqlabel{rLiedef}
[x, \, y]_{r} := [x, \, y] + x \triangleleft r(y) - y
\triangleleft r(x)
\end{equation}
is a Lie algebra called the $r$-deformation of $\mathfrak{h}$.
Furthermore, $\mathfrak{h}_{r} \cong \widetilde{\mathfrak{h}}$, as
Lie algebras.
\end{theorem}

\begin{proof}
$(1)$ To start with, we will prove that $\widetilde{\mathfrak{h}}
= \{\bigl(r(x) \oplus x \bigl) ~|~ x \in \mathfrak{h}\}$ is a Lie
subalgebra of $\Xi = \mathfrak{g} \bowtie \mathfrak{h}$. Indeed,
for all $x$, $y \in \mathfrak{h}$ we have:
\begin{eqnarray*}
\bigl[r(x) \oplus x, \, r(y) \oplus y\bigl]
&\stackrel{\equref{bracketmpLie}}{=}& \bigl(\underline{\bigl[r(x),
\, r(y)\bigl] + x \triangleright r(y) - y \triangleright
r(x)}\bigl) \,\, \oplus \\
&& \bigl([x, \, y] + x \triangleleft r(y) - y \triangleleft
r(x)\bigl)\\
&\stackrel{\equref{factLie}}{=}& r ([x, \, y] + x \triangleleft
r(y) - y \triangleleft r(x))\, \oplus \, \bigl([x, \, y] + x
\triangleleft r(y) - y \triangleleft r(x)\bigl)
\end{eqnarray*}
Moreover, it is straightforward to see that $\mathfrak{g} \cap
\widetilde{\mathfrak{h}} = \{0\}$ and $a \oplus x = \bigl(a -
r(x)\bigl) \oplus \bigl(r(x) \oplus x \bigl) \in \mathfrak{g} +
\widetilde{\mathfrak{h}}$. Therefore, $\widetilde{\mathfrak{h}}$
is a $\mathfrak{g}$-complement of $\Xi$.

$(2)$ We denote by $\widetilde{f_{r}}$ the $k$-linear isomorphism
from $\mathfrak{h}$ to $\widetilde{\mathfrak{h}}$ induced by
$f_{r}$. We will prove that $\widetilde{f_{r}}$ is also a Lie
algebra map if we consider on $\mathfrak{h}$ the bracket given by
\equref{rLiedef}. Indeed, for any $x$, $y \in \mathfrak{h}$ we
have:
\begin{eqnarray*}
\widetilde{f_{r}}\bigl([x,\,
y]_{r}\bigl)&\stackrel{\equref{rLiedef}}{=}&
\widetilde{f_{r}}\bigl([x, \, y] + x \triangleleft r(y) - y
\triangleleft
r(x)\bigl)\\
&{=}& \underline{r \bigl([x, \, y]\bigl) + r (x \triangleleft
r(y)) - r(y \triangleleft r(x))} \oplus [x, \, y] + x
\triangleleft r(y) - y \triangleleft r(x)\\
&\stackrel{\equref{factLie}}{=}& [r(x), \, r(y)] + x
\triangleright r(y) - y \triangleright r((x) \oplus [x, \, y] + x
\triangleleft r(y) - y \triangleleft r(x)\\
&\stackrel{\equref{bracketmpLie}}{=}& [r(x) \oplus x, \, r(y)
\oplus y] = [\widetilde{f_{r}}(x), \, \widetilde{f_{r}}(y)]
\end{eqnarray*}
Therefore, $\mathfrak{h}_{r}$ is a Lie algebra and the proof is
now finished.
\end{proof}

We are now able to describe all complements of a Lie subalgebra
$\mathfrak{g}$ of $\Xi$.

\begin{theorem} \thlabel{descrierecomlie}
Let $\mathfrak{g}$ be a Lie subalgebra of $\Xi$, $\mathfrak{h}$ a
given $\mathfrak{g}$-complement of $\Xi$ with the associated
canonical matched pair of Lie algebras $(\mathfrak{g},
\mathfrak{h}, \triangleright, \triangleleft)$. Then
$\overline{\mathfrak{h}}$ is a $\mathfrak{g}$-complement of $\Xi$
if and only if there exists an isomorphism of Lie algebras
$\overline{\mathfrak{h}} \cong \mathfrak{h}_{r}$, for some
deformation map $r: \mathfrak{h} \to \mathfrak{g}$ of the matched
pair $(\mathfrak{g}, \mathfrak{h}, \triangleright,
\triangleleft)$.
\end{theorem}

\begin{proof}
Let $\overline{\mathfrak{h}}$ be an arbitrary
$\mathfrak{g}$-complement of $\Xi$. Since $\Xi = \mathfrak{g}
\oplus \mathfrak{h} = \mathfrak{g} \oplus \overline{\mathfrak{h}}$
we can find four $k$-linear maps:
$$
u: \mathfrak{h} \to \mathfrak{g}, \quad
v: \mathfrak{h} \to \overline{\mathfrak{h}},
\quad t:\overline{\mathfrak{h}} \to \mathfrak{g},
\quad w: \overline{\mathfrak{h}} \to \mathfrak{h}
$$
such that for all $x \in \mathfrak{h}$ and $y \in
\overline{\mathfrak{h}}$ we have:
\begin{equation} \eqlabel{lie111}
x = u(x) \oplus v(x), \qquad y = t(y) \oplus w(y)
\end{equation}
By applying \equref{lie111} for $x = w(y) \in \mathfrak{h}$, $y
\in \overline{\mathfrak{h}}$, we get:
\begin{eqnarray*}
-t(y) \oplus y = w(y) \stackrel{\equref{lie111}}{=}
u\bigl(w(y)\bigl) \oplus v\bigl(w(y)\bigl)
\end{eqnarray*}
Therefore, by the unique decomposition in a direct sum, we obtain
$v\bigl(w(y)\bigl) = y$ and $u\bigl(w(y)\bigl) = -t(y)$, for all
$y \in \overline{\mathfrak{h}}$. In the same manner it can be
proved that $w\bigl(v(x)\bigl) = x$ and $t\bigl(v(x)\bigl) =
-u(x)$, for all $x \in \mathfrak{h}$. In particular, we proved
that $v: \mathfrak{h} \to \overline{\mathfrak{h}}$ is a $k$-linear
isomorphism. We denote by $\tilde{v}: \mathfrak{h} \to
\mathfrak{g} \bowtie \mathfrak{h}$ the composition:
$$
\tilde{v} : \, \mathfrak{h} \, \stackrel{v} {\longrightarrow} \,
\overline{\mathfrak{h}} \, \stackrel{i}{\hookrightarrow} \, \Xi \,
= \,\mathfrak{g} \bowtie \mathfrak{h}
$$
More precisely, we have $\tilde{v}(x) = v(x)
\stackrel{\equref{lie111}}{=} -u(x) \oplus x$, for all $x \in
\mathfrak{h}$. We denote $r := -u$ and we will prove that $r$ is a
deformation map and $\overline{\mathfrak{h}} \cong
\mathfrak{h}_{r}$. Indeed, $\overline{\mathfrak{h}} = {\rm Im} (v)
= {\rm Im} (\tilde{v})$ is a Lie subalgebra of $\Xi = \mathfrak{g}
\bowtie \mathfrak{h}$ and therefore we have:
\begin{eqnarray*}
[r(x) \oplus x, \, r(y) \oplus y]
&\stackrel{\equref{bracketmpLie}}{=}& \bigl([r(x), \, r(y)] + x
\triangleright r(y) - y \triangleright r(x)\bigl) \,\, \oplus \\
&& \bigl([x, \, y] + x \triangleleft r(y) - y \triangleleft
r(x)\bigl) \\
&=& r(z) \oplus z
\end{eqnarray*}
for some $z \in \mathfrak{h}$. Thus, we obtain:
\begin{equation}\eqlabel{lie113}
r(z) = [r(x), \, r(y)] + x \triangleright r(y) - y \triangleright
r(x), \qquad z = [x, \, y] + x \triangleleft r(y) - y
\triangleleft r(x)
\end{equation}

By applying $r$ to the second part of \equref{lie113} we get:
\begin{eqnarray*}
r(z) = r([x, \, y] ) + r(x \triangleleft r(y)) - r(y \triangleleft
r(x)) = [r(x), \, r(y)] + x \triangleright r(y) - y \triangleright
r(x)
\end{eqnarray*}
Therefore, $r$ is a deformation map of the matched pair
$(\mathfrak{g}, \mathfrak{h}, \triangleright, \triangleleft)$.
Moreover, we have:
$$
[v(x), \, v(y)] = v(z) \stackrel{\equref{lie113}}{=} v \bigl([x,
\, y] + x \triangleleft r(y) - y \triangleleft r(x)\bigl)
\stackrel{\equref{rLiedef}}{=} v([x, \, y]_{r})
$$
that is, $v: \mathfrak{h}_{r} \to \overline{\mathfrak{h}}$ is a
Lie algebra map and the proof is now finished.
\end{proof}

In order to classify all complements we need to introduce the
following:

\begin{definition}\delabel{equivLie}
Let $(\mathfrak{g}, \mathfrak{h}, \triangleright, \triangleleft)$
be a matched pair of Lie algebras. Two deformation maps $r$, $R:
\mathfrak{h} \to \mathfrak{g}$ are called \emph{equivalent} and we
denote this by $r \sim R$ if there exists $\sigma: \mathfrak{h}
\to \mathfrak{h}$ a $k$-linear automorphism of $\mathfrak{h}$ such
that for any $x$, $y\in \mathfrak{h}$:
\begin{equation}\eqlabel{equivLiemaps}
\sigma \bigl([x, \, y]\bigl) - \bigl[\sigma(x), \, \sigma(y)\bigl]
= \sigma(x) \triangleleft R \bigl(\sigma(x)\bigl) - \sigma\bigl(x
\triangleleft r(y)\bigl) - \sigma(y) \triangleleft R
\bigl(\sigma(x)\bigl) + \sigma \bigl(y \triangleleft r(x)\bigl)
\end{equation}
\end{definition}

As a conclusion of this section we obtain the answer to the (CCP)
for Lie algebras:

\begin{theorem}\thlabel{clasformelorLie}
Let $\mathfrak{g}$ be a Lie subalgebra of $\Xi$, $\mathfrak{h}$ a
$\mathfrak{g}$-complement of $\Xi$ and $(\mathfrak{g},
\mathfrak{h}, \triangleright, \triangleleft)$ the associated
canonical matched pair. Then:

$(1)$ $\sim$ is an equivalence relation on ${\mathcal D}{\mathcal
M} \, (\mathfrak{h}, \mathfrak{g} \, | \, (\triangleright,
\triangleleft) )$. We denote by  ${\mathcal H}{\mathcal A}^{2}
(\mathfrak{h}, \mathfrak{g} \, | \, (\triangleright,
\triangleleft) )$ the quotient set ${\mathcal D}{\mathcal M} \,
(\mathfrak{h}, \mathfrak{g} \, | \, (\triangleright,
\triangleleft) )/ \sim$.

$(2)$ There exists a bijection between the isomorphism classes of
all $\mathfrak{g}$-complements of $\Xi$ and ${\mathcal H}{\mathcal
A}^{2} (\mathfrak{h}, \mathfrak{g} \, | \, (\triangleright,
\triangleleft) )$. In particular, the factorization index of
$\mathfrak{g}$ in $\Xi$ is computed by the formula:
$$
[\Xi : \mathfrak{g}]^f = | {\mathcal H}{\mathcal A}^{2}
(\mathfrak{h}, \mathfrak{g} \, | \, (\triangleright,
\triangleleft) )|
$$
\end{theorem}

\begin{proof}
It follows from \thref{descrierecomlie} that in order to classify
all $\mathfrak{g}$-complements of $\Xi$ it is enough to consider
only $r$-deformations of $\mathfrak{h}$, for various deformation
maps $r : \mathfrak{h} \to \mathfrak{g}$. Now let $r$, $R :
\mathfrak{h} \to \mathfrak{g}$ be two deformation maps. As
$\mathfrak{h}_r = \mathfrak{h}_R : = \mathfrak{h}$ as $k$-spaces,
we obtain that the Lie algebras $\mathfrak{h}_r$ and
$\mathfrak{h}_R$ are isomorphic if and only if there exists
$\sigma : \mathfrak{h}_{r} \to \mathfrak{h}_{R}$ a $k$-linear
isomorphism which is also a Lie algebra map. Using
\equref{rLiedef} we obtain that $\sigma$ is a Lie algebra map if
and only if the compatibility condition \equref{equivLiemaps}
holds, i.e. $r \sim R$. Hence, $r \sim R$ if and only if $\sigma:
\mathfrak{h}_r \to \mathfrak{h}_R$ is an isomorphism of Lie
algebras. Thus we obtain that $\sim$ is an equivalence relation on
${\mathcal D}{\mathcal M} (\mathfrak{h}, \mathfrak{g} \, | \,
(\triangleright, \triangleleft) )$ and the map
$$
{\mathcal H}{\mathcal A}^{2} (\mathfrak{h}, \mathfrak{g} \, | \,
(\triangleright, \triangleleft) ) \to {\mathcal F} (\mathfrak{g},
\Xi),  \qquad \overline{r} \mapsto \mathfrak{h}_{r}
$$
is a bijection between sets, where $\overline{r}$ is the
equivalence class of $r$ via the relation $\sim$.
\end{proof}

\begin{example}
Let $k$ be a field of ${\rm char} (k) \neq 2$ and $\Xi$ the
$4$-dimensional Lie algebra with $\{e_1, e_2, f_1, f_2 \}$ as a
basis and the bracket given by:
$$
[e_1, e_2] = 2 e_1, \,\,\, [e_1, f_1] = e_2, \,\,\, [e_2, f_2] = 2
f_1
$$
Let $\mathfrak{g}$ be the Lie subalgebra of $\Xi$ with basis
$\{e_1, e_2 \}$. Then $[\Xi : \mathfrak{g}]^f = 2$.

Indeed, let $\mathfrak{h}$ be the abelian Lie algebra of dimension
$2$ with basis $\{f_1, f_2\}$. Then $\mathfrak{h}$ is a
$\mathfrak{g}$-complement of $\Xi$ with the associated canonical
matched $(\mathfrak{g}, \, \mathfrak{h}, \, \triangleleft, \,
\triangleright)$ given as follows:
$$
f_1 \triangleright e_1 := - e_2, \quad f_1 \triangleleft e_2 := -2
f_1
$$
It is straightforward to see that $r_{c}: \mathfrak{h} \to
\mathfrak{g}$ given by $r(f_1) := 0$, $r(f_2) := c f_2$, for some
$c\in k$ is a deformation map of the matched pair $(\mathfrak{g},
\mathfrak{h}, \triangleright, \triangleleft)$. Furthermore, the
$r_c$-deformation of $\mathfrak{h}$ has the bracket $[f_1, \,
f_2]_{r_{c}} := - 2 c f_1$. As this is not an abelian Lie algebra
for $c \neq 0$, it follows that $\mathfrak{h}_{r_c}$ is not
isomorphic to $\mathfrak{h}$. Since there are only two types of
Lie algebras of dimension $2$ we obtain that $[\Xi :
\mathfrak{g}]^f = 2$.
\end{example}

\section{Outlooks and open problems} \selabel{probnoi}
In this paper we solve the (CCP) for the category of Lie algebras
and Hopf algebras respectively. The common tool used is the
bicrossed product construction. All the results proven above can
serve as a model for obtaining similar theories for the (CCP) in
other categories $\Cc$ where the bicrossed product was introduced,
such as: (co)algebras, $C^*$-algebras or von Neumann algebras, Lie
groups, locally compact groups or locally compact quantum groups,
groupoids or quantum grupoids, multiplier Hopf algebras, etc.
Another direction for further inquiry is given by the following
three open questions related to the results of this work:

\textbf{Question 1:} \emph{Let $\tau : H \ot H \to k$ be a
normalized Sweedler cocycle and $H_{\tau}$ be Doi's \cite{Doi}
deformation of the Hopf algebra $H$. Does there exist a Hopf
algebra $A$, a matched pair of Hopf algebras $(A, H,
\triangleleft, \triangleright)$ and a deformation map $r : H \to
A$ such that $H_{\tau} = H_r$, where $H_r$ is the $r$-deformation
of $H$ in the sense of \thref{deformatANM}?}

Having in mind \thref{clasformelor} it is natural to ask:

\textbf{Question 2:} \emph{Does there exist an example of an
extension of finite dimensional Hopf algebras $A \subset E$ having
an infinite factorization index $[E : A]^f$?}

We have approached the (CCP) for right $A$-complements of a given
extension $A \subset E$ of Hopf algebras. The same theory can be
developed for left $A$-complements. If $E$ has a bijective
antipode then $H$ is a right $A$-complement if and only if $H$ is
a left $A$-complement (\cite[Proposition 3.1]{abm1}); i.e. in this
case $[E : A]_r^f = [E : A]_l^f$, where $[E : A]_r^f$ (resp. $[E :
A]_r^f$) denotes the factorization index corresponding to right
(resp. left) $A$-complements. In general, we ask the following:

\textbf{Question 3:} \emph{Does there exist an example of an
extension $A \subset E$ of Hopf algebras such that $[E : A]_r^f
\neq [E : A]_l^f$?}

\section*{Acknowledgment}
We are deeply indebted to the referee for his/her very important
comments which shortened some of the arguments previously used.
Nevertheless, the direct approach used in the case of Lie algebras
was also his/her suggestion.


\begin{thebibliography}{99}

\bibitem{abm1}
Agore, A.L., Bontea, C.G., Militaru, G. - Classifying bicrossed
products of Hopf algebras, {\sl Algebr. Represent. Theory}, DOI:
10.1007/s10468-012-9396-5, arXiv:1205.6110.

\bibitem{am-2011}
Agore, A.L., Militaru, G. - Extending structures II: the quantum
version, {\sl J. Algebra} {\bf 336} (2011), 321--341.

\bibitem{am-2012}
Agore, A.L., Militaru, G. - Classifying complements for groups.
Applications, arXiv:1204.1805.

\bibitem{AA}
Aguiar, M.,  Andruskiewitsch, N. - Representations of matched
pairs of groupoids and applications to weak Hopf algebras.
Algebraic structures and their representations, {\sl Contemp.
Math.}, {\bf 376} (2005), 127--173, Amer. Math. Soc., Providence.

\bibitem{cap}
Cap, A., Schichl, H., Vanzura, J. - On twisted tensor product of
algebras, {\sl Comm. Algebra}, {\bf 23} (1995), 4701 -- 4735.

\bibitem{DeDW}
Delvaux, L., Van Daele, A., Wang, S.H. - Bicrossproducts of
multiplier Hopf algebras, {\sl J. Algebra}, {\bf 343} (2011), 11
-- 36.

\bibitem{Doi}
Doi, Y. - Braided bialgebras and quadric bialgebras,  {\sl Comm.
Algebra}, {\bf 21} (1993), 1731 -- 1749.

\bibitem{DT2}
Doi, Y., Takeuchi, M. - Multiplication alteration by two-cocycles,
{\sl Comm. Algebra}, {\bf 22} (1994), 5715 -- 5732.

\bibitem{dr}
Drinfel'd, V. G. - Quantum Groups, in {\sl Proceedings ICM,
Berkeley}, 1986, 798--820.

\bibitem{LPS2}
Liebeck, M.W. , Praeger, C. E., Saxl, J. - Regular subgroups of
primitive permutation groups, {\sl Mem. Amer. Math. Soc.} {\bf
203} (2010), no. 952.

\bibitem{Kassel}
Kassel, C. - Quantum groups, Graduate Texts in Mathematics 155.
Springer-Verlag, New York, 1995.

\bibitem{KO}
Knus, M. A., Ojanguren, M. - Th\'eorie de la descente et
alg\`ebres d'Azumaya, {\sl Lecture Notes in Math.} {\bf 389},
Springer Verlag, Berlin, 1974.

\bibitem{majid3}
Majid, S. - Matched Pairs of Lie Groups and Hopf Algebra
Bicrossproducts, {\sl Nuclear Physics B, Proc. Supl.} {\bf 6}
(1989), 422--424.

\bibitem{majid}
Majid, S. - Physics for algebraists: non-commutative and
non-cocommutative  Hopf algebras by a bicrossproduct construction,
{\sl J. Algebra}, {\bf 130} (1990), 17--64.

\bibitem{majid2}
Majid, S. - Foundations of quantum groups theory, Cambridge
University Press, 1995.

\bibitem{Masuoka}
Masuoka, A. - Extensions of Hopf algebras and Lie bialgebras, {\sl
Trans. Amer. Math. Soc.}, {\bf 352} (2000), 3837--3879.

\bibitem{Mo}
Molnar, R.K. - Semi-direct products of Hopf algebras, {\sl J.
Algebra} {\bf 47} (1977),  29 -- 51.

\bibitem{parker}
Parker, D.B. - Forms of coalgebras and Hopf algebras, {\sl J.
Algebra} {\bf 239} (2001),  1 -- 34.

\bibitem{VV}
Vaes, S., Vainerman, L. - Extensions of locally compact quantum
groups and the bicrossed product construction,  {\sl Adv. Math.}
{\bf 175} (2003), no. 1, 1--101.

\bibitem{WW}
Wiegold, J., Williamson, A. G. - The factorisation of the
alternating and symmetric groups. {\sl Math. Z.} {\bf 175} (1980),
no. 2, 171--179.

\end{thebibliography}
\end{document}